\documentclass[a4paper]{article}
\usepackage{amsfonts,amsmath,amssymb,amsthm}
\usepackage{color}
\usepackage{ifpdf}
\ifpdf
    \usepackage[pdftex]{graphicx}
    \DeclareGraphicsExtensions{.png,.pdf,.jpg}
\else
   \usepackage[dvips]{graphicx}
   \DeclareGraphicsExtensions{.eps}
\fi
\graphicspath{{.}{figuresPL/}}

\usepackage{graphicx}

\newtheorem{lemma}{Lemma}[section]
\newtheorem{theorem}[lemma]{Theorem}

\newtheorem{proposition}[lemma]{Proposition}
\newtheorem{remark}[lemma]{Remark}
\newtheorem{definition}[lemma]{Definition}

\def\Om{\Omega}
\def\S{\mathbb{S}}
\def\R{\mathbb{R}}
\def\T{\mathbb{T}}
\def\E{\mathcal{E}}
\renewcommand{\S}{\mathbb{S}}
\def\Z{\mathbb{Z}}
\def\N{\mathbb{N}}

\renewcommand{\phi}{\varphi}
\def\p{\phi}

\def\Div{\textup{div}\,}

\def\HH{\mathcal{H}^{d-1}}
\def\1{\mathbf{1}}

\def\CA{\mathcal{CA}}

\def\XXint#1#2#3{{\setbox0=\hbox{$#1{#2#3}{\int}$ }
\vcenter{\hbox{$#2#3$ }}\kern-.57\wd0}}

\def\X{X}
\def\eps{\varepsilon}
\renewcommand{\subset}{\subseteq}

\newcommand{\mesrest}{\text{\huge$\llcorner$}}

\begin{document}

\title{Existence and qualitative properties of isoperimetric sets in periodic media}
\author{A. Chambolle \footnote{CMAP, Ecole Polytechnique, CNRS,
        Palaiseau, France, email: antonin.chambolle@cmap.polytechnique.fr},
\and
        M. Goldman
        \footnote{Max Planck Institute for Mathematics in the Sciences, Leipzig, Germany, email: goldman@mis.mpg.de, funded by a Von Humboldt PostDoc fellowship}
        \and M. Novaga
        \footnote{Dipartimento di Matematica, Universit\`a di Padova,
        via Trieste 63, 35121 Padova, Italy, email: novaga@math.unipd.it}
}
\date{}
\maketitle
\begin{abstract}
We review and extend here some recent results on the existence of minimal surfaces and isoperimetric sets in non homogeneous and anisotropic  periodic media. 
We also describe the qualitative properties of the homogenized surface tension, also known as stable norm (or minimal action) in Weak KAM theory.  
In particular we investigate its strict convexity and differentiability properties.
\end{abstract}

\section{Introduction}
In Euclidean spaces, it is well known that hyperplanes are local minimizers of the perimeter and that balls are the (unique) solutions to the isoperimetric problem i.e. 
they have the least perimeter among all the sets having a given volume. The situation of course changes for interfacial energies which are no longer homogeneous
nor isotropic but it is still natural to investigate the existence of local minimizers which are plane-like and of compact isoperimetric sets in this context. 
More precisely, for an open set $\Om\subset \R^d$ and a set of finite perimeter $E$ (see \cite{giusti}), we will consider interfacial energies of the form
\[\E(E,\Om):=\int_{\partial^* E\cap \Om} F(x,\nu^E) d\HH\]
where $\HH$ is the $(d-1)$-dimensional  Hausdorff measure,  $\nu^E$ is  the internal normal to $E$, $\partial^*E$ is the reduced boundary of $E$, and $F(x,p)$ is continuous and periodic in $x$, convex and one-homogeneous in $p$ with
\begin{equation}\label{boundF} 
 c_0|p|\le F(x,p)\le c_0^{-1}|p| \qquad \forall (x,p)\in \R^d\times \R^d
\end{equation}
for some $c_0\in (0,1]$. When $\Om=\R^d$, we will simply denote by $\E$, the functional $\E(\cdot,\R^d)$. In the following we will denote by $\T$
 the $d$ dimensional torus and let $Q:=[0,1)^d$. 

In a first part, we review the fundamental result of  Caffarelli and De La Llave \cite{CDLL} concerning the existence of plane-like minimizers of $\E$ and we will
 define a homogenized energy $\p(p)$ (usually called the stable norm or the minimal action functional), which represents the average energy of a plane-like minimizer in the direction $p$.  The qualitative  properties of the minimal action are studied in the second section.  The following result was proven in \cite{CGN_PL} (see also 
 \cite{auerbangert,HJG}). 
\begin{itemize}
\item If $p$ is ``totally irrational'' (meaning that there exists no $q\in \Z^d$ such that $q\cdot p= 0$)  then $\nabla\p(p)$
exists.
\item The same occurs for any $p$ such that the plane-like minimizers satisfying the strong Birkhoff property  give rise to a foliation of the space.
\item If there is a gap in this lamination and if
 $(q_1, \dots , q_k)\in \Z^d$ is a maximal
family of independent vectors such that $q_i\cdot p=0$, then
$\partial \phi(p)$ is a convex set of dimension $k$, and
$\phi$ is differentiable in the directions which are orthogonal
to $\{q_1, \dots,q_k\}$. In particular if $p$ is not totally irrational then $\p$ is not differentiable at $p$.
\item  $\phi^2$ is strictly convex.
\end{itemize}

In the last section, we extend some results of \cite{GNvol} concerning the existence of compact minimizers of the isoperimetric problem
\begin{equation}\label{Probisoper}
 \min_{|E|=v} \E(E)
\end{equation}
for every given volume $v>0$ and show that these minimizers, once rescaled, converge to the Wulff shape associated to the stable norm $\p$.\\

Let us conclude this introduction by pointing out that, using a deep result of Bourgain and Brezis \cite{bourgainbrezis}, see also \cite{DPPF,CT, CGN_PL}, all the results presented here directly extend to functionals of the form
\begin{equation}\label{funcg}
 \int_{\partial^* E\cap \Om} F(x,\nu^E) d\HH+\int_{E\cap \Om} g(x)\ dx
\end{equation}
where $g\in L^d(\T)$ is a periodic function with zero mean satisfying some smallness assumption (for the results of Section \ref{secreg} to hold, one needs also that $g$ is Lipschitz continuous).\\
Notice also that when considering the perimeter i.e. when $F(x,p)=|p|$,  smooth minimizers of \eqref{funcg} satisfy the prescribed mean curvature equation
\[\kappa_E =-g\]
where $\kappa_E$ is the mean curvature of the set $E$. The existence of plane-like minimizers of \eqref{funcg} can then be rephrased in term of existence of plane-like sets with prescribed mean curvature. On the other hand, in \cite{GNvol},
 the isoperimetric problem \eqref{Probisoper} was introduced in order to study existence of compact sets with prescribed mean curvature, leading to the proof of the following theorem.
\begin{theorem}
 Let $d\le 7$ and $g$ be a periodic ${C}^{0,\alpha}$ function on $\R^d$ with zero average and satisfying a suitable smallness assumption.
 Then for every $\eps>0$ there exists $\eps'\in [0,\eps]$ such that there exists a compact solution of 
 \[
 \kappa_E=g+\eps'.
 \]
\end{theorem}

\section{Plane-like minimizers}
In \cite{CDLL}, Caffarelli and De La Llave proved the existence of plane-like minimizers of $\E$. 
\begin{theorem}\label{thmcdll}
There exists $M>0$ depending only on $c_0$  such that for every $p\in \R^d\setminus\{0\}$ and $a\in \R$, there exists a local minimizer (also called Class A Minimizer) $E$ of $\E$ such that  
\begin{equation}\label{PLexist}
\left\{ x\cdot \frac{p}{|p|} > a+M\right\}
\ \subset \ E\ \subset \ \left\{ x\cdot \frac{p}{|p|} > a-M\right\}.
\end{equation}
Moreover $\partial E$ is connected. A set satisfying the condition \eqref{PLexist} is called plane-like. 
\end{theorem}
\begin{definition}
Let $p\in \R^d\setminus\{0\}$ and let $E$ be a plane-like minimizer of $\E$ in the direction $p$. We set  
\[\p(p):= |p|\lim_{R\to \infty} \frac{1}{\omega_{d-1} R^{d-1}} \, \E(E, B_R),\]
where $\omega_{d-1}$ is the volume of the unit ball in $\R^{d-1}$.
\end{definition}
{ } Caffarelli and De La Llave  proved that this limit exists and does not depend on $E$. In \cite{CT}, the first author and Thouroude related this
definition to the cell formula: 
\begin{equation}\label{defphi}
\p(p)\ =\ \min\left\{
\int_{\T} F(x,p+D v(x))\,:\,
v\in BV(\T)
\right\}\,.
\end{equation}
It is obvious from~\eqref{defphi} that $\p$ is a convex, one-homogeneous function. However, since the problem defining $\phi$ is not strictly convex, in general the minimizer of \eqref{defphi} is not unique.
Nevertheless, this uniqueness generically holds (see \cite[Th. 4.23, Th. B.1]{CGN_PL}). This is an instance of the so-called Ma\~n\'e's conjecture.  It has been  shown in~\cite{CT}
that the minimizers of \eqref{defphi} give an easy way to construct plane-like minimizers:
\begin{proposition}
Let $v_p$ be a minimizer of \eqref{defphi} then for every $s\in \R$, the set $\{v_p(x)+ p\cdot x >s\}$ is a plane-like minimizer of $\E$ in the direction $p$.
\end{proposition}
{ } 
For $\eps>0$ and $E\subset\R^d$ of finite perimeter, let 
\[
\E_\eps(E) := \eps^{(d-1)}\E\left(\eps^{-1}E\right)
= \int_{\partial^* E} F\left(x/\eps,\nu^E\right) \ d\HH.
\]  
It was shown in \cite{CT} (see also \cite{BrCP}) that the convergence of the average energy of plane-like minimizers to the stable norm can also be reinterpreted in term of $\Gamma$-convergence \cite{DM}.

\begin{theorem}\label{thconv}
 When $\eps\to 0$, the functionals $\E_\eps$ $\Gamma$-converge,
with respect to the $L^1$-convergence of the characteristic functions, to the anisotropic functional
\[
\E_0(E)= \int_{\partial^*E}\phi(\nu^E)\,d\mathcal H^{d-1}
\qquad E\subset\R^d\ {\rm of\ finite\ perimeter.}
\]
\end{theorem}

\section{Strict convexity and differentiability properties of the stable norm}\label{secreg}
In this section we are going to study the differentiability and strict convexity of the stable norm $\phi$. It is a geometric analog of the minimal action functional of KAM theory whose differentiability has first been studied by Aubry and Mather \cite{aubry,mather} for geodesics on the two dimensional torus. The results of Aubry and Mather  have then been extended by Moser \cite{moser2}, in the framework of non-parametric integrands,  
and more recently by Senn \cite{S}. In this context, the study of the set of non-self-intersecting minimizers, which correspond to our plane-like minimizers satisfying the Birkhoff property has been performed by Moser and Bangert \cite{moser,Bangertmin}, 
whereas the proof of the strict convexity of the minimal action has been recently shown by Senn \cite{Sennstrict}. Another related problem is the homogenization of periodic Riemannian metrics (geodesics are objects of dimension one whereas in our problem the hypersurfaces are of codimension one). We refer to \cite{burago,BrButFra} for more information on this problem.

 We define the polar function of $F$ by $$F^\circ(x,z):=\sup_{\{F(x,p)\le1\}} z\cdot p$$ so that $(F^\circ)^\circ=F$. If we denote by $F^*(x,z)$ the convex conjugate of $F$ with respect to the second variable then $\{F^*(x,z)=0\}=\{F^\circ \le 1\}$. 
We will make the following additional hypotheses on $F$:
\begin{itemize}
\item $F$ is $C^{2,\alpha}(\R^d\times (\R^d\setminus\{0\}))$, 
\item $F$ is elliptic (that is $F(x,p)-C|p|$ is a convex function of $p$ for some $C>0$).
\end{itemize}
With these hypothesis we have \cite{ASS,Duzaar} \cite[Prop. 3.4]{CGN_PL}
\begin{proposition}\label{maxprinc}
For any plane-like Class A Minimizer $E$, the reduced boundary $\partial^* E$ is of class $C^{2,\alpha}$ and $\mathcal{H}^{d-3} (\partial E \setminus \partial^* E)=0$.
Let $E_1\subset E_2$ be two Class A Minimizers with connected boundary, 
then $\mathcal{H}^{d-3} (\partial E_1 \cap \partial E_2)=0$.
\end{proposition}

In the following we let
\[
\X:=\{ z \in L^{\infty}(\T,\R^d) \,:\, \Div z=0 \ , \ F^\circ(x,z(x))\le 1 \ a.e \}.
\]
Using arguments of convex duality, it is possible to characterize the stable norm as a support function \cite[Prop. 2.13]{CGN_PL}.
\begin{proposition}\label{propdual}
 There holds
\[\p (p)=\sup_{z \in \X} \left(\int_{\T} z\right)\cdot p\,.\]

Hence, the subgradient of $\p$ at $p\in\R^d$ is given by
\begin{equation}
\partial\p (p)\ =\ \left\{ \int_{\T} z \,:\, z \in \X \, ,\,  \left(\int_{\T} z\right)\cdot p=\p(p)\right\}.
\end{equation}
 \end{proposition}
 
Since
\[\p \textrm{ is differentiable at } p \quad  \Longleftrightarrow \quad \partial \p (p) \textrm{ is a singleton},\]
Proposition \ref{propdual} tells us that checking the differentiability of $\phi$ at a given point $p$ is equivalent to checking whether 
for any  vectorfields $z_1,z_2\in X$,
\[\left(\int_{\T} z_1\right)\cdot p=\left(\int_{\T} z_2\right)\cdot p=\p(p) \qquad  \Longrightarrow \qquad \int_{\T} z_1 =\int_{\T} z_2. \]

We now introduce the notion of calibration. 
\begin{definition}
 We say that a vector field $z \in \X$ is a periodic calibration of a set $E$ of locally finite perimeter if, we have
\[ [z,\nu^E]= F(x,\nu^E) \qquad \HH\mesrest \partial^*E-a.e.\]
When no confusion can be made, by calibration we mean a periodic calibration.
\end{definition}
In the previous definition, $[z,\nu^E]$ has to be understood in the sense of Anzellotti but is roughly speaking
 $z(x)\cdot \nu^E(x)$ when it makes sense (see \cite{Anzellotti,CGN_reg}). By the differentiability of $F(x,\cdot)$, this implies that on a calibrated set, the value of $z$ is imposed since (see \cite{CGN_reg} for a more precise statement)
\begin{equation}\label{equalz}
z(x)=\nabla_p F(x,\nu^E)
\end{equation}
Using some arguments of convex analysis and the coarea formula, it is possible to prove the following relation between calibrations and minimizers of \eqref{defphi}.
\begin{proposition}
Let $z\in X$ a vector field such that
$\left(\int_\T z\right) \in \partial \p(p)$.
Then, for any minimizer $v_p$ of~\eqref{defphi},
\begin{equation}\label{EL}
[z, Dv_p+p]\ =\ F(x,Dv_p+p)\qquad |Dv_p+p|-a.e.
\end{equation}
and for every $s\in \R$, $z$ calibrates the set $E_s := \{ v_p + p\cdot x > s\}$. We say that such a vector field $z$ is a calibration in the direction $p$.  
\end{proposition}
Equation \eqref{EL} is the Euler-Lagrange equation associated to \eqref{defphi}. Notice that thanks to \eqref{equalz}, the value of any calibration in the direction $p$ is fixed on  $\partial E_s$. 
Hence, it is reasonable to expect that if these sets fill a big portion of the space, the average on the torus of any calibration will be fixed which would
 imply the differentiability of $\phi$.  One of the important consequences of calibration is that it implies an ordering of the plane-like minimizers.
\begin{proposition}\label{order}
Suppose that $z\in \X$ calibrates two  plane-like 
minimizers $E_1$ and $E_2$ with connected boundaries. Then, either $E_1\subset E_2$, or $E_2\subset E_1$. As a consequence $\mathcal{H}^{d-3}(\partial E_1  \cap \partial E_2)=0$.
\end{proposition}
Using the cell formula we can already prove the strict convexity of $\phi^2$. 
\begin{theorem}\label{convex}  
The function $\phi^2$ is strictly convex.
\end{theorem}

\begin{proof}
Let $p_1,p_2$, with $p_1\neq p_2$, and  
let $p=p_1+p_2$. We want to show that, if $\p(p)=\p(p_1)+\p(p_2)$, then 
$p_1$ is proportional to $p_2$, which gives the thesis. 

Indeed, we have
\begin{align*}
\p(p) &= \int_\T [z_p, p+Dv_p]\\
&= \, \int_\T F(x,p+Dv_{p})\\
&\le\, \int_\T F(x,p+Dv_{p_1}+Dv_{p_2}) \\
&\le\, \int_\T F(x,p_1+Dv_{p_1})+F(x,p_2+Dv_{p_2})\\ 
&=\, \p(p_1)+\p(p_2)\,.
\end{align*}
Since $\p(p)=\p(p_1)+\p(p_2)$, it follows that $v_{p_1}+v_{p_2}$ is also a minimizer of \eqref{defphi} and,
in particular, $z_p$ satisfies
\[
[z_p,p_1+Dv_{p_1}]+[z_p,p_2+Dv_{p_2}]\ =\ 
F(x,p_1+Dv_{p_1})+F(x,p_2+Dv_{p_2})
\]
$(|p_1+Dv_{p_1}|+|p_2+Dv_{p_2}|)$-a.e., so that 
\[
[z_p, p_i+Dv_{p_i}]\ =\ F(x,p_i+Dv_{p_i}) \qquad i\in\{1,2\}\,.
\]
This means that $z_p$ is a calibration for the plane-like minimizers 
\[
\{v_p +p\cdot x\ge s\}\ ,\ \{v_{p_1}+p_1 \cdot x\ge s\}\ \textup{ and }\ \{v_{p_2} +p_2\cdot x\ge s\}
\]
for all $s\in\R$. By Proposition \ref{order}, it follows that they are included in one another
which is possible only if $p_1$ is proportional to $p_2$.
\end{proof}

We can also show that $\phi$  is differentiable in the totally irrational directions.
 
\begin{proposition}\label{totir}
Assume $p$ is totally irrational.
Then for any two calibrations $z,z'$ in the direction $p$,
$\int_\T z\,dx=\int_\T z'\,dx$. As a consequence, 
$\partial \p(p)$ is a singleton and $\p$ is differentiable at $p$.
\end{proposition}
\begin{proof}
The fundamental idea is that since $p$ is totally irrational, even if the levelsets $\{v_p+p\cdot x>s\}$ do not fill the whole space, the remaining holes must have finite volume and therefore do not count in the average.\\

 Consider $z,z'$ two calibrations and a solution $v_p$ of~\eqref{defphi}.

{ } Let us show that, for any $s$,
\begin{equation}
\int_{\{v_p+p\cdot x=s\}} (z(x)-z'(x))\,dx\ =\ 0.
\end{equation}
Let $C_s:=\{x\,:\, v_p(x)+p\cdot x=s\}$ then
$\partial C_s = \partial \{v_p+p\cdot x>s\}\cup \partial \{v_p+p\cdot x\ge s\}$.
Moreover, all the $C_s$ have zero Lebesgue measure except for a countable number of values.
Consider such a value $s$. Since $z$ and $z'$ calibrate $C_s$ which is a   plane-like minimizer we have $[z,\nu^{C_s}]=[z',\nu^{C_s}]$ on $\partial^* C_s$.

{ } Then, we observe that the sets $C_s^q=Q\cap (C_s-q)$, $q\in\Z^d$, are
all disjoint since $p$ is totally irrational and since all the $C_s^q$ are calibrated by $z$, so that their measures sum up to less than $1$.

{ }Let $e_i$ be a vector of the canonical basis of $\R^d$ then by the divergence Theorem (where the integration by parts can by justified thanks to $|C_s|\le1$) we compute
\begin{equation}\label{psiRR}
\int_{C_s}(z(x)-z'(x))\cdot e_i\,dx
\,=\, -\int_{C_s} x_i \, \Div(z(x)-z'(x))\,dx=0 \ ,
\end{equation}
which gives our claim.

In particular, we obtain that $\int_{\R^d}(z-z')\,dx =\sum_{s} \int_{C_s} (z-z')\,dx=0$ hence $\int_Q (z-z')\,dx=0$.
\end{proof}
When $p$ is not totally irrational, we have to consider a slightly bigger class of plane-like minimizers than those obtained as $\{v_p+p\cdot x>s\}$ for $v_p$ a minimizer of \eqref{defphi} and $s\in \R$. Indeed, we must consider all the plane-like minimizers which are maximally periodic. 
\begin{definition}\label{defBirk}
Following \cite{HJG,S,UNIBan} we give the following definitions:
\begin{itemize}
\item we say that $E\subset\R^d$ satisfies the Birkhoff property if,
for any $q\in\Z^d$, either $E\subseteq E+q$ or $E+q\subseteq E$;
\item we say that $E$ satisfies the strong Birkhoff property
in the direction $p\in\Z^d$ if $E\subseteq E+q$ when $p\cdot q\le 0$
and $E+q\subseteq E$ when $p\cdot q\ge 0$.
\end{itemize}
  We will let $\CA(p)$ be the set of all the plane-like minimizers in the direction $p$ which satisfy the strong Birkhoff property
\end{definition}
Notice that the sets $\{v_p+p\cdot x>s\}$ have the strong Birkhoff property. It can be shown that the sets of this form correspond exactly to the recurrent plane-like minimizers which are those which can be
 approximated by below or by above by entire translations of themselves (see \cite[Prop. 4.18]{CGN_PL}). For sets satisfying the Birkhoff property there holds \cite[Lem. 4.13, Prop. 4.14, Prop 4.15]{CGN_PL}.
\begin{proposition}
 Let $E$ be a Class A minimizer with the Birkhoff property: then it is a plane-like minimizer (with a constant $M$ just depending on $c_0$ and $d$), calibrated and $\partial E$ is connected.
\end{proposition}

 For sets satisfying the Strong Birkhoff property, it can be further proven \cite[Th. 4.19]{CGN_PL}.

\begin{theorem}\label{calibrallPL}
Let  $z$ be a calibration in the direction $p$,
then $z$ calibrates every plane-like minimizer
with the strong Birkhoff property.
\end{theorem}
As a consequence of Proposition \ref{order} and Theorem \ref{calibrallPL}, for every $p\in \R^d\setminus\{0\}$, the   plane-like minimizers of $\CA(p)$ form a lamination of $\R^d$ (possibly with gaps). In light of \eqref{equalz}, we see that
\begin{proposition}
 If there is no gap in the lamination by plane-like minimizers of $\CA(p)$ then $\phi$ is differentiable at the point $p$.
\end{proposition}

We are thus just left to prove that if $p$ is not totally irrational (meaning that there exists $q\in\Z^d$ such that $p\cdot q=0$) and if 
there is a gap $G$ (whose boundary $\partial E^+\cup \partial E^-$ is made of two plane-like minimizers of $\CA(p)$) in the lamination then
 $\phi$ is not differentiable at the point $p$. To simplify the notations and the argument, let us consider
 the case $p=2$ and $(p,q)=(e_1,e_2)$, the canonical basis of $\R^2$. Let $E_n:=\{v_{e_1+\frac{1}{n}e_2}+(e_1+\frac{1}{n}e_2)\cdot x>0\}$ be plane-like
 minimizers in the direction $e_1+\frac{1}{n}e_2$ which intersects $G$, then up to translations (in the direction $e_2$), we can assume that there is a subsequence 
which converges to a set $H_+$ which also intersects the gap. It can be shown that $H_+$ is an heteroclinic solution meaning 
that it is included inside $G$, satisfies the Birkhoff property (but not the  strong one), and that $H_+ \pm ke_2\to E^{\pm}$ when $k\in \N$ 
goes to infinity (see \cite[Prop. 4.27]{CGN_PL}). Moreover $H_+$ is calibrated by $z_+:=\lim z_n$ where $z_n$ is any calibration in the direction $ e_1+\frac{1}{n}e_2$ (notice that $z_+$ is then a calibration in the direction $e_1$). 
Consider similarly $H_-$ (respectively $z_-$), an heteroclinic solution in the direction $-e_2$ ( respectively a calibration of $H_-$) then we aim at proving that 
\[\int_{G\cap Q} [z_+-z_-, e_2]>0\]
which would imply the non differentiability of $\phi$ at $e_1$ (in the direction $e_2$). 
  \begin{proposition}
For $t\in [0,1)$, let $S_t:=\{x\cdot e_2=t\}$ (and $S=S_0$) then almost every $s, t \in \R$, we have
\begin{equation}\label{segt}
\int_{S_t\cap G} \left[z_+-z_-, e_2\right]=\int_{S_s\cap G} \left[z_+-z_-, e_2\right].
\end{equation}
In particular,
\begin{equation}\label{egsurf}\int_{Q\cap G} [z_+-z_-, e_2] = \int_{S\cap G} \left[z_+-z_-,e_2\right].
\end{equation}
\end{proposition}

\begin{proof}
Fix $s<t \in \R$, let $S^t_s:=\left\{ x \in Q \,:\, s< x \cdot e_2 < t\right\}$  then 
\begin{align*}
 0=\int_{S_s^t\cap G} \Div  (z_+-z_-) &	=\int_{S_t\cap G}  [z_+-z_-, e_2]-\int_{S_s\cap G} [z_+-z_-, e_2].
\end{align*}
\end{proof}

\begin{proposition}
Let $\nu$ be the inward normal to $H_q$, then
 \begin{equation}\label{egheter}\int_{S\cap G} \left[z_+-z_-, e_2\right]= 
 \int_{\partial^* H_+} [z_+-z_-, \nu].\end{equation}

\end{proposition}
\begin{proof}
We first introduce some additional notation (see Figure \ref{sigmaplus}): let 
\[\Sigma^+:= \partial^* H_+ \cap \{ x \cdot e_2 >0\} \qquad G^+:= G\cap \{ x \cdot e_2 >0\} \cap H_+^c\]
and 
\[S^+:= S\cap G \cap H_+^c.\]

\begin{figure}[ht]
\centering{\input{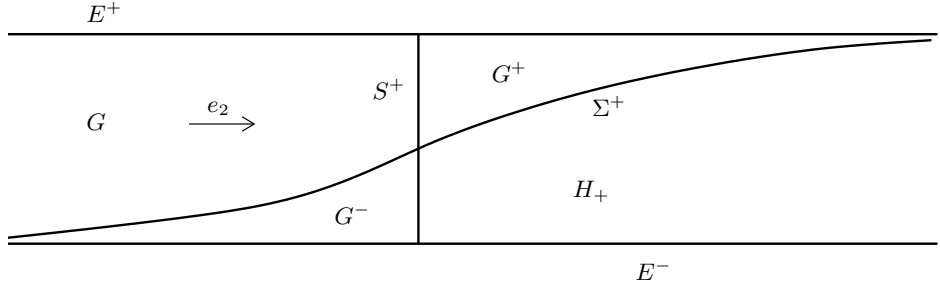}}
\caption{Heteroclinic solution in the direction $e_2$.}
\label{sigmaplus}
\end{figure}

Then, 
\[
 0=\int_{G^+} \Div  (z_+-z_-)=\int_{S^+} \left[z_+-z_-, e_2\right] -\int_{\Sigma^+} [z_+-z_-, \nu]
\]

Similarly we define $\Sigma^-$ and $S^-$ and get 
\[\int_{\Sigma^-} [z_+-z_-, \nu]= \int_{S^-} \left[z_+-z_-, e_2\right].\]
Summing these two equalities we find \eqref{egheter}.
\end{proof}
We can now conclude. Indeed, since $z_+$ calibrates $H_+$ and since $z_-\in \X$, on $\partial^* H_+$, there holds $[z_+-z_-,e_2]=F(x,\nu)-[z_-,\nu]\ge 0$ and thus if 
\[\int_{\partial^* H_+} [z_+-z_-, \nu]=0\]
then $F(x,\nu)=[z_-,\nu]$ on $\partial H_+$ and thus $z_-$ calibrates also $H_-$ which would lead to a contradiction since it implies that $H_+$ and $H_-$ cannot cross. Hence, 
\[\int_{G\cap Q} [z_+-z_-, e_2]=\int_{\partial^* H_+} [z_+-z_-, \nu]>0.\]
In conclusion we have (see \cite{CGN_PL} for a complete proof)

\begin{proposition}
 If there is a gap in the lamination  by plane-like minimizers of $\CA(p)$ and if
 $(q_1, \dots , q_k)\in \Z^d$ is a maximal
family of independent vectors such that $q_i\cdot p=0$, then
$\partial \phi(p)$ is a convex set of dimension $k$, and
$\phi$ is differentiable in the directions which are orthogonal
to $\{q_1, \dots,q_k\}$. In particular if $p$ is not totally irrational then $\p$ is not differentiable at $p$.
\end{proposition}

\section{Existence and asymptotic behavior of isoperimetric sets} 

In this section we extend some results of \cite{GNvol} on the existence of compact minimizers of the isoperimetric problem \eqref{Probisoper}. In addition to \eqref{boundF}, we will make the hypothesis that $F$ is uniformly Lipschitz continuous in $x$, i.e. there exists $C>0$ such that
\[|F(x,p)-F(y,p)|\le C|x-y| \qquad \qquad \forall (x,y,p)\in \R^d\times \R^d\times \S^{d-1}.\]
 Our strategy will differ from the one of \cite[Th. 2.6]{GNvol}. It will instead closely follow \cite[Th. 4.9]{GNvol}.
 The idea is to use the Direct Method of the calculus of variations together with a kind of concentration compactness argument to deal with the 
invariance  by translations of the problem. Notice that a similar strategy has been used to prove existence of minimal clusters (see \cite[Th. 29.1]{maggi}).
We first recall Almgren's Lemma (see \cite[Lem. II.6.18]{maggi}, \cite{morgan}).
\begin{lemma}\label{almgren}
If $E$ is a set of finite perimeter and $A$ is an open set of $\R^d$ such that $\HH(\partial^*E\cap A)>0$ then there exists $\sigma_0>0$ and $C>0$ such that for every $\sigma\in (-\sigma_0,\sigma_0)$ there exists a set $F$ such that 
\begin{itemize}
 \item $F\Delta E\Subset A$,
\item $|F|=|E|+\sigma$,
\item $|\E(F,A)-\E(E,A)|\le C |\sigma|$.
\end{itemize}
\end{lemma}
 
We now prove that any minimizer (if it exists) has to be compact.
\begin{proposition}\label{bounddiam}
 For every $v>0$, every minimizer $E$ of \eqref{Probisoper} has bounded diameter. 
\end{proposition}

\begin{proof}
The proof follows the classical method to prove density estimates for minimizers of isoperimetric problems 
(see for instance \cite{giusti}). Fix $v>0$ and let $E$ be a minimizer of \eqref{Probisoper}. Let then $f(r):=|E\backslash B_r|$. Let us assume that the diameter of $E$ is not finite, so that $f(r)>0$ for every $r>0$. 
Without loss of generality, we can also assume that $\HH(\partial^* E\cap B_1)>0$.
Let $\sigma_0$ and $C$ be given by Lemma \ref{almgren} with $A=B_1$, and fix $R>1$ such that $f(R)\le \sigma_0$ then for every $r>R$ there exists $F$ such that 
\begin{itemize}
\item $F\Delta E\Subset B_1$,
\item $|F|=|E|+f(r)$,
\item $|\E(E,B_r)-\E(F,B_r)|\le C f(r)$.
\end{itemize}
 Letting $G:= F\cap B_r$ we have $|G|=|E|$ thus, by minimality of $E$, we find
\[\E(E)\le \E(G)\le \E(F,\overline B_r)+c_0^{-1}\HH(\partial B_r\cap F)\le \E(E,\overline B_r)+C f(r)+ c_0^{-1}\HH(\partial B_r\cap E)\]
and thus
\[
c_0 \HH(\partial^*E\backslash \overline B_r)\le 
\E(E, \overline B_r^c) \le C f(r) + c_0^{-1}\HH(\partial B_r\cap E).
\]
Recalling that $f'(r)=- \HH(\partial B_r\cap E)$ and $\HH(\partial^*E\cap \partial B_r)=0$ for a.e. $r>0$,
we get
\[
c_0 \HH(\partial^*(E\backslash B_r)) = 
c_0 \HH(\partial^*E\backslash \overline B_r) - c_0 f'(r)\le 
C f(r)- \left(c_0+c_0^{-1}\right)f'(r)
\]
for a.e. $r>0$.
By the isoperimetric inequality \cite{giusti} it then follows
\[c_1 f(r)^{\frac{d-1}{d}}\le c_2 f(r)-f'(r)\]
for some constants $c_1,c_2>0$.
If now $R_1>R$ is such that $f(R_1)^\frac 1 d\le \frac{c_1}{2c_2}$, we get for $r\ge R_1$,
\[\frac{c_1}{2} f(r)^{\frac{d-1}{d}}\le -f'(r)\]
and thus $\left(f^{1/d}\right)'\le -\frac{c_1}{2}$, which leads to a contradiction.
\end{proof}

\begin{remark}\rm
 Adapting the proof of \cite{cinti} to the anisotropic case, it should be possible to prove the boundedness of minimizers under the weaker assumption that $F(\cdot,p)$ is continuous (using the so-called $\eps-\eps^\beta$ property).
\end{remark}

We can now prove the existence of compact minimizers for every volume $v>0$.
\begin{theorem}\label{existvol}
 For every $v>0$ there exists a compact minimizer of \eqref{Probisoper}.
\end{theorem}

\begin{proof}
 To simplify  the notations, let us assume that $v=1$. Let $E_k$ be a minimizing sequence meaning that $|E_k|=1$ and $\E(E_k)\to \inf_{|E|=1} \E(E)$. For every $k\in \N$, let  
$\{Q_{i,k}\}_{i\in\mathbb N}$ be a partition of  $\R^d$ into disjoint cubes of equal volume larger than $2$, such that 
the sets $E_{k}\cap Q_{i,k}$ are of decreasing measure, and 
let  $x_{i,k}=|E_{k}\cap Q_{i,k}|$. By the isoperimetric inequality, 
there exist $0<c<C$ such that 
\begin{align*}
c\sum_i x_{i,k}^\frac{d-1}{d}&=c\sum_i \min\left(|E_{k}\cap Q_{i,k}|,|Q_{i,k} \backslash E_{k}|\right)^\frac{d-1}{d}\\
			   & \le \sum_i P(E_{k},Q_{i,k})\\			   
   			   &\le \sum_i c_0 \E(E_k,Q_{i,k})\\
			   &\le c_0 \E(E_k)\le C
\end{align*}
hence
\[
\sum_{i=1}^{\infty} x_{i,k}=1 \qquad 
{\rm and}\qquad 
\sum_{i=1}^{\infty} x_{i,k}^\frac{d-1}{d}\le \frac{C}{c}.
\]
Since $x_{i,k}$ is nonincreasing with respect to $i$, it follows that (cf \cite[Lem. 4.2]{GNvol}) for any $N$
 \begin{equation}\label{ndelta2}
 \sum_{i=N}^{\infty} x_{i,k} \ \le\ \frac{C}{c} \frac{1}{N^{1/d}}\,.
 \end{equation}

Up to extracting a subsequence, we can suppose that  
$x_{i,k}\to \alpha_i\in [0,1]$ as $k\to +\infty$ for every $i\in\mathbb N$, so that by \eqref{ndelta2} we have
\begin{equation}\label{eqstella}
\sum_i \alpha_i=1.
\end{equation}
Let $z_{i,k} \in Q_{i,k}$. Up to extracting a further subsequence, we can suppose that $d(z_{i,k},z_{j,k}) \to c_{ij} \in [0,+\infty]$, and

\[
\left(E_{k}-z_{i,k}\right) \to E_i \quad \textrm{ in the } L^1_{\rm loc}\textrm{-convergence}
\]
for every $i\in\mathbb N$. And it is not very difficult to check that $E_i$ are minimizers of \eqref{Probisoper} under the volume constraint $v_i:=|E_i|$. Notice that by Proposition \ref{bounddiam}, each $E_i$ is bounded.

We say that $i \sim j$ if $c_{ij} < +\infty$ and we denote by $[i]$ the equivalence class of $i$.
Notice that $E_i$ equals $E_j$ up to a translation, if $i\sim j$. We want to prove that
\begin{equation}\label{rodi}
\sum_{[i]} v_i = 1,
\end{equation}
where the sum is taken over all equivalence classes.
For all $R>0$ let $Q_R=[-R/2,R/2]^d$ be the cube of sidelength $R$. Then for every $i\in\mathbb N$,
\[
|E_i| \geq |E_i \cap Q_R|= \lim_{k\to +\infty} \left|\left(E_{k} -z_{i,k}\right) \cap Q_R\right|.
\]
If $j$ is such that $j \sim i$ and $c_{ij} \le \frac{R}{2}$, possibly increasing $R$ we have 
$Q_{j,k}- z_{i,k} \subset Q_R$ for all $k\in\mathbb N$, so that
\[
\lim_{k\to +\infty} \left|\left(E_{k} -z_{i,k}\right) \cap Q_R\right|\geq \lim_{k \to +\infty} 
\sum_{c_{ij} \leq \frac{R}{2}} |E_{k} \cap Q_{j,k}|=\sum_{c_{ij} \leq \frac{R}{2}} \alpha_j.
\]
Letting $R\to +\infty$ we then have
\[
|E_i| \geq \sum_{i\sim j} \alpha_j 
\]
hence, recalling \eqref{eqstella},
\[
\sum_{[i]} |E_i| \ge 1,
\]

thus proving \eqref{rodi} (since the other inequality is clear).

Let us now show that 
\begin{equation}\label{rofinal}
 \sum_{[i]} \E(E_i) \le \inf_{|E|=1} \E(E).
\end{equation}
Choosing a representative in each equivalence class $[i]$ and reindexing, from now on 
we shall assume that $c_{ij}=+\infty$ for all $i\ne j$.
Let $I \in \mathbb{N}$ be fixed. Then for every $R>0$ there exists $K \in \mathbb{N}$ such that for every $k\ge K$ and $i$, $j$ less than $I$, we have 
\[d(z_{i,k},z_{j,k}) > R. 
\]
For $k \ge K$ we thus have
\begin{align*}
 \E(E_{k})\geq & \sum_{i=1}^I \int_{\partial E_{k} \cap (B_R+z_{i,k})} F(x,\nu^{E_k})\,d\mathcal H^{d-1}
 \\
		=   & \sum_{i=1}^I \int_{\partial (E_{k}-z_{i,k}) \cap B_R} F(x,\nu^{E_k})\,d\mathcal H^{d-1}
		\\
		=   & \sum_{i=1}^I \E(E_{k} -z_{i,k} ,B_R)
\end{align*}

{}From this, and the lower-semicontinuity of $\E$, we get
\[
 \inf_{|E|=1} \E(E)\ge \sum_{i=1}^I \liminf_{k\to \infty} \E(E_{k} -z_{i,k} ,B_R) \ge \sum_{i=1}^I \E(E_i, B_R).
\]
Letting $R\to \infty$  and then $I\to \infty$ (if the number of equivalence classes is finite then just take $I$ equal to this number), we find \eqref{rofinal}. Let finally $d_i:= \textup{diam}(E_i)$ and $F:=\bigcup_{i} \left(E_i+ 2d_i e_1\right)$ where $e_1$ is a unit vector then $|F|=1$ and 
\[\E(F)=\sum_{i} \E(E_i)\le \inf_{|E|=1} \E(E)\]
and thus $F$ is a minimizer of \eqref{Probisoper} (notice that by Proposition \ref{bounddiam}, we must have $E_i=\emptyset$ for $i$ large enough).

\end{proof}

\begin{remark}\label{remspadaro}\rm
 Another proof, in the spirit of \cite[Th. 2.6]{GNvol} would consist in proving first existence of compact minimizers of the relaxed problems
\begin{equation}\label{Emu}
 \min_{E\subset \R^d} \E(E)+\mu ||E|-v|
\end{equation}
for $\mu>0$ using uniform density estimates (see \cite[Prop. 2.1, Prop. 2.3]{GNvol}) and then showing that for $\mu$ large enough, the minimizers of \eqref{Emu}
 have volume exactly equal to $v$. In this more general situation with respect to the one studied in \cite{GNvol}, instead of relying on the Euler-Lagrange equation as in \cite[Th. 2.6]{GNvol} (which works only in a smooth setting i.e. for $F$ elliptic and smooth and for low dimension $d$)
 one could argue by contradiction and follow the lines of \cite{FE}. Notice also that contrary to \cite[Th. 2.6]{GNvol}, this strategy (just as the one adopted here in the proof of Theorem \ref{existvol}) would not give quantitative bounds on the diameter.
\end{remark}
\begin{remark}\rm
 The isoperimetric problem \eqref{Probisoper} is very similar to the isoperimetric problem on manifold with densities which has 
recently attracted a lot of attention and where similar issues of existence of compact minimizers appear (see \cite{morgan, MP, cinti}). 
Notice however that in these works, the media is usually considered as isotropic, meaning that $F(x,p)=f(x)|p|$ with some hypothesis on the behavior at infinity (or with some radial symmetry) of $f$ 
which is not compatible with periodicity. 
\end{remark}

\begin{remark}\rm
 Using Almgren's Lemma, it is not difficult to see that minimizers of the isoperimetric problem \eqref{Probisoper}
 are quasi-minimizers of $\E$ (of course without volume constraint anymore) and as such, they enjoy the same regularity properties
 (see \cite[Example 2.13]{maggi},\cite{Duzaar}). In particular, under the hypothesis of Section \ref{secreg}, they are $C^{2,\alpha}$ out of a singular set of $(d-3)$-Hausdorff measure equal to zero. 
\end{remark}

Let $W=\{\phi^\circ\le 1\}$ be the Wulff shape associated to $\phi$. It is then the (unique) solution to the isoperimetric problem associated to $\phi$ (see \cite{FM})
\[\min_{|E|=|W|} \int_{\partial^* E} \phi(\nu^E) d \HH.\]
 For $v>0$, let $E_v$ be a compact minimizer of \eqref{Probisoper}. Let $\eps:= \left(\frac{|W|}{v}\right)^{1/d}$ and $E_\eps:= \eps E_v$ then $E_\eps$ is a minimizer of 
\[\min_{|E|=|W|} \int_{\partial^* E} F(x/\eps,\nu^E) d\HH\]
then using Theorem \ref{thconv} and following the same proof as in Theorem \ref{existvol} (see \cite[Th. 4.9]{GNvol}), we get:
\begin{theorem}
 There exist a sequence of vectors $z_\eps\in \R^d$ such that $E_\eps+z_\eps \to W$ when $\eps\to 0$.
\end{theorem}

The asymptotic shape for small volume has been  investigated in a very precise way in \cite{figmag}.
\section*{Acknowledgments.}
The second author thanks E.~Spadaro and E.~Cinti for interesting discussions about Almgren's Lemma and Remark \ref{remspadaro}.

\end{document}